%% file: main.tex
\documentclass[11pt,a4paper]{article}
\usepackage[left=0.5in, right=0.5in, top=0.5in, bottom=0.5in]{geometry}

\usepackage{graphicx}%
\usepackage{multirow}%
\usepackage{amsmath,amssymb,amsfonts}%
\usepackage{mathrsfs}%
\usepackage[title]{appendix}%
\usepackage{xcolor}%
\usepackage{textcomp}%
\usepackage{manyfoot}%
\usepackage{booktabs}%
\usepackage{algorithm}%
\usepackage{algorithmicx}%
\usepackage{algpseudocode}%
\usepackage{listings}%

\usepackage{mathtools} 
\usepackage{bm} 
\usepackage{enumitem} 
\usepackage{orcidlink} 
\usepackage{authblk}
\usepackage{hyperref}

\include{definitions} 

\newtheorem{theorem}{Theorem}
\newtheorem{proposition}[theorem]{Proposition}%
\newtheorem{example}{Example}%
\newtheorem{remark}{Remark}%

\newtheorem{definition}{Definition}%

\title{Sectional Kolmogorov $N$-widths for parameter-dependent function spaces: A general framework with application to parametrized Friedrichs' systems}

\author[1]{Christian Engwer\,\orcidlink{0000-0002-6041-8228} }
\author[1]{Mario Ohlberger\,\orcidlink{0000-0002-6260-3574}}
\author[1,2,3]{\underline{Lukas Renelt}\,\orcidlink{0009-0003-3161-5219}
\footnote{Corresponding author: lukas.renelt@inria.fr}}

\affil[1]{%
  University of M\"unster,
  Institute for Analysis and Numerics,
	Einsteinstr.\ 62,
	M\"unster,
	48149,
	Germany}

\affil[2]{%
	INRIA,
	48 Rue Barrault,
	75647 Paris,
	France}
	
\affil[3]{%
	CERMICS,
	Ecole nationale des ponts et chauss\'ees, IP Paris,
	77455 Marne la Vall\'ee,
	France}

\begin{document}
	
\maketitle	
	
\begin{abstract}
  We investigate parametrized variational problems
  where for each parameter the solution may originate from a different parameter-dependent function space.
  Our main motivation is the theory of Friedrichs' systems, a large abstract class of
  linear PDE-problems whose solutions are sought in
  operator- (and thus parameter-)dependent graph spaces.
  Other applications include function spaces on parametrized domains or discretizations involving data-dependent stabilizers.
  Concerning the set of all parameter-dependent solutions, we argue that in these cases the interpretation as a ``solution manifold'' widely adopted in the model order reduction community is no longer applicable.
  Instead, we propose a novel framework based on the theory of fiber bundles and explain how established concepts such as approximability generalize by introducing a Sectional Kolmogorov $N$-width.
  Further, we prove exponential approximation rates of this $N$-width if a norm equivalence criterion is fulfilled.
  Applying this result to problems with Friedrichs' structure then gives a sufficient criterion that can be easily verified.
\end{abstract}

\section{Introduction}\label{sec:introduction}
In the following we are interested in general variational problems of the form
\begin{equation}\label{eq:generalVariationalProblem}
  \FindT u_\mu\in\trialSpace\colon\quad
  b_\mu(u_\mu,v) = f_\mu(v) \qquad\forallT v\in\testSpace
\end{equation}
where a parameter $\mu\in\Pcal\subset\R^p$ may influence both the (bi-)linear forms $b_\mu(\cdot,\cdot)$ and $f_\mu(\cdot)$, e.g.\ in the form of physical parameters, as well as the trial and test spaces $\trialSpace$ and $\testSpace$. For the $\mu$-dependency in the function spaces, one should think either of a parametrization of the functions regularity, a dependency in the norms or a geometrical parametrization of the underlying domain and/or its boundary.
\par
In applications where parametrized problems such as~\cref{eq:generalVariationalProblem} need to be solved for many different parameters, using high-dimensional discretizations such as the finite element method (FEM) for each individual parameter ceases to be viable. In the case of parameter-independent function spaces, model order reduction (MOR) has proven itself to efficiently address this challenge by first determining a \emph{reduced approximation space} from high-dimensional evaluations which subsequently allows for a fast computation of reduced solutions for arbitrary parameters, we refer to~\cite{
	quarteroni2015reducedBasis,hesthaven2016reducedBasis,benner2017modelReduction} for introductory works.
\par
Considering parametrized function spaces as in~\cref{eq:generalVariationalProblem}, mainly geometric parametrizations $\trialSpace=X(\Omega_\mu)$ have been investigated, see for example~\cite{thiyagarajan2005shapeOptimizationRB,lassila2010parametric,manzoni2012shape} for contributions in the context of shape optimization.
Parametrizations of the regularity, let alone the general setting~\cref{eq:generalVariationalProblem}, have (at least to our knowledge) not been analyzed thus far, potentially due to a perceived lack of applications requiring such a general approach.
The necessity of a generalized theory becomes particularly apparent if one considers so called \emph{Friedrichs' systems}~\cite{
	friedrichs1958,ernGuermond2007,ernGuermond2006friedrichs1,ernGuermond2006friedrichs2,ernGuermond2008friedrichs3,antonic2009graphSpaces,antonic2010intrinsic,antonic2011boundaryOps,burazin2016nonStationaryFS,antonic2017complexFS}
whose solutions originate from function spaces with the graph-structure
\begin{equation}
	\trialSpace \;=\; \left\{u\in\LtwoM \;:\; A_\mu u \in\LtwoM\right\}
\end{equation}
where $A_\mu$ denotes a pametrized first order differential operator with certain additional properties (a rigorous definition will be given in~\cref{sec:friedrichs}).
\par
A second application are discretizations of problems with hyperbolic character. Focusing on the linear problem~\cref{eq:generalVariationalProblem}, this amounts to problems including a dominating advection term that needs to be properly handled. Discretizations such as the Streamline Upwind Petrov-Galerkin (SUPG) method, ensure discrete stability by including additional stabilizing components into the function spaces. In the case of SUPG one for example pairs a given trial space $X$ with the test space
\begin{equation}
	\testSpace \coloneqq \left\{v + \tau_\mu \vec{b}_\mu\nabla v \;:\; v\in X\right\}
\end{equation}
where $\vec{b}_\mu$ is the advection field and $\tau_\mu$ denotes a data- (and thus parameter-) dependent stabilization parameter. Stabilization by the computation of supremizers falls into a similar category, considering for example the Discontinuous Petrov-Galerkin (DPG) method with optimal test functions~\cite{demkowiczDPG1} or the approach presented in~\cite{DahmenHuangSchwab}, both resulting in a parameter-dependent test space $\testSpace$ if the underlying PDE is parametrized.
In \cite{romor2025}, model order reduction for Fridrichs' systems has been discussed with localization based on the Discontinuous Galerkin method.
\par
A third research field where the full order model involves parameter-dependent spaces, are spectral multiscale methods. Here, the data functions involve fluctuations on a micro scale $h$ which is too small to be resolved by conventional finite element discretizations. However, the micro scale variations still influence the global solution behavior significantly and thus can also not be omitted. One approach of dealing with such problems is to solve the problem on a tractable macro scale $H\gg h$ and use problem-adapted basis functions which are usually obtained by solving local problems, see e.g.\ \cite{babuska1994GFEM}. Such a generalized finite element space on a coarse mesh $\Omega_H$ can then be of the form
\begin{equation}
	\trialSpace \coloneqq \linSpan\left\{\varphi_T\cdot\left(u_{\mu}^T\right)\restr{T} \;:\; T\in \Omega_H\right\}
\end{equation}
where $\varphi_T$ is a macro-scale partition-of-unity and each $u_{\mu}^T$ solves a local problem on a fine-scale neighborhood $\mathcal{N}_h(T)$ of the macro element $T$.

\par
These three examples should already make it clear that parametrized function spaces are not just a mere academic curiosity. While the case of solely parameter-dependent test spaces $\testSpace$ has already been discussed in~\cite{renelt2023model}, the case of parametrized trial spaces is much more interesting from a theoretical point of view. One of the main contributions of this paper is a new way of interpreting solutions from different function spaces in a differential geometric way - similar to the way that model order reduction is commonly interpreted as the approximation of the \emph{solution manifold} consisting of all parameter-dependent solutions. Before continuing this thought in more detail, let us briefly give an outline of the paper:
\par
In~\cref{sec:diffGeoFramework} we introduce our abstract setting and discuss the relation to the concepts currently prevalent in the MOR community.
The section-based perspective is detailed in~\cref{sec:fibers,sec:sectionalApprox} and followed by a discussion of the various ways a parameter may influence the trial space. In~\cref{sec:parameterIndependent} we first show that the parameter-independent case is consistent with the manifold-perspective.
Subsequently, we formalize the further distinction into parametrizations varying the regularity~(\cref{sec:regularityParam}) and geometrical parametrizations~(\cref{sec:geometricalParametrization}) and show an approximability result for the former.
\Cref{sec:friedrichs} then focuses on Friedrichs' systems, first introducing the general concept and subsequently applying the results from~\cref{sec:diffGeoFramework}. This then allows for a classification of various concrete parametrized PDEs permitting a reformulation as a Friedrichs' system.

\section{A differential geometric framework for parametrized solution spaces}\label{sec:diffGeoFramework}
Let us begin by formally defining the abstract problem~\cref{eq:generalVariationalProblem} we will serve as the base for any further considerations. To be precise, we are interested in solutions to the problem
\begin{equation}\tag{\ref{eq:generalVariationalProblem}}
	\FindT u_\mu\in\trialSpace\colon\quad
	b_\mu(u_\mu,v) = f_\mu(v) \qquad\forallT v\in\testSpace
\end{equation}
where $\trialSpace$ and $\testSpace$ are Banach-spaces of $\R^m$-valued functions on a bounded domain $\Omega\subseteq\R^d$ and equipped with potentially parameter-dependent norms $\norm{\cdot}_{\trialSpace}$ and $\norm{\cdot}_{\testSpace}$, respectively. As per usual, the set of all parameters $\Pcal\subseteq\R^p$ is assumed to be compact. Let us briefly mention that all of our considerations will also be valid for $\Pcal$ being a $p$-dimensional compact Banach-manifold. We further assume that for any fixed parameter $\mu$ the problem is well-posed. We want to emphasize that~\cref{eq:generalVariationalProblem} also allows for finite-dimensional problems e.g.\ obtained through discretization of an underlying infinite-dimensional problem.
Under these prerequisites, we are interested in approximating the set of all possible solutions which we denote by
\begin{equation}\label{def:solutionSet}
  \Mcal \;\coloneqq\; \{u_\mu \;:\quad \mu\in\Pcal,\;
  u_\mu\in\trialSpace \;\textnormal{solves}\; \cref{eq:generalVariationalProblem} \}.
\end{equation}
In classic model order reduction theory, all solutions are assumed to originate from the same function space $X$ and one can under minor assumptions classify the solution set $\Mcal$ as a compact sub-manifold of $X$. The core idea of model order reduction is then to approximate $\Mcal$ as a whole by a simpler structure which can be evaluated quickly. A ``linear approximation'' subsequently amounts to finding a linear subspace $X_N\subset X$ of dimension $N$ and a reduced solution is given by projecting onto this subspace. In order to measure the approximation quality of a given $N$-dimensional subspace, one usually compares to the Kolmogorov $N$-width
\begin{equation}\label{eq:kolmogorovNWidth}
  \kolmogorov{\Mcal}
  \coloneqq
  \Inf_{\substack{X_N \subset X \;\textnormal{linear} \\ \dim(X_N)\,\leq\, N}}\;
  \sup_{u_\mu\in\Mcal}\; \Inf_{u_N\in X_N} \norm{u_\mu-u_N}_{X}
\end{equation}
which amounts to the lowest projection error (always considering the worst approximated element) one can reach with an $N$-dimensional subspace. While the optimal space is usually unobtainable, there exist methods that construct quasi-optimal spaces, i.e.\ spaces whose projection error decays with the same rate as the Kolmogorov $N$-width. On the other hand, a slowly decreasing $N$-width means that \emph{any} linear method will not perform well.
We refer to the review \cite{devore2017foundation} for a further discussion of the Kolmogorov $N$-width in the context of model order reduction and generalizations to nonlinear $N$-widths as well as to the alternative concept of entropy numbers for compact sets. 
\par
These considerations demonstrate how central the perspective of approximating a manifold by linear subspaces is to the development of model order reduction methods. However, one notes that these concepts no longer apply to the set~\cref{def:solutionSet} where each element originates from a different function space $\trialSpace$. Similarly, it is not clear how~\cref{eq:kolmogorovNWidth} translates and what ``linear approximation'' actually means. In the following, we propose an alternative viewpoint and show the connections to the aforementioned perspective.

\subsection{Fiber bundles}\label{sec:fibers}
To that end, let us further analyze the structure of the set $\totalSpace\coloneqq \bigcup_{\mu\in\Pcal}X_\mu$ which contains the solution set $\Mcal$. We argue that the structure of $\totalSpace$ resembles that of a \emph{fiber bundle} which is a generalization of the concept of \emph{vector bundles}. Intuitively one may picture these as a collection of vector spaces (called \emph{fibers}) ``glued together'' at the respective zero elements (which are by definition contained in every  vector space), see~\cref{fig:generalBundle}¸ for a visualization.

\begin{figure}
	\centering
	\includegraphics[width=0.5\linewidth]{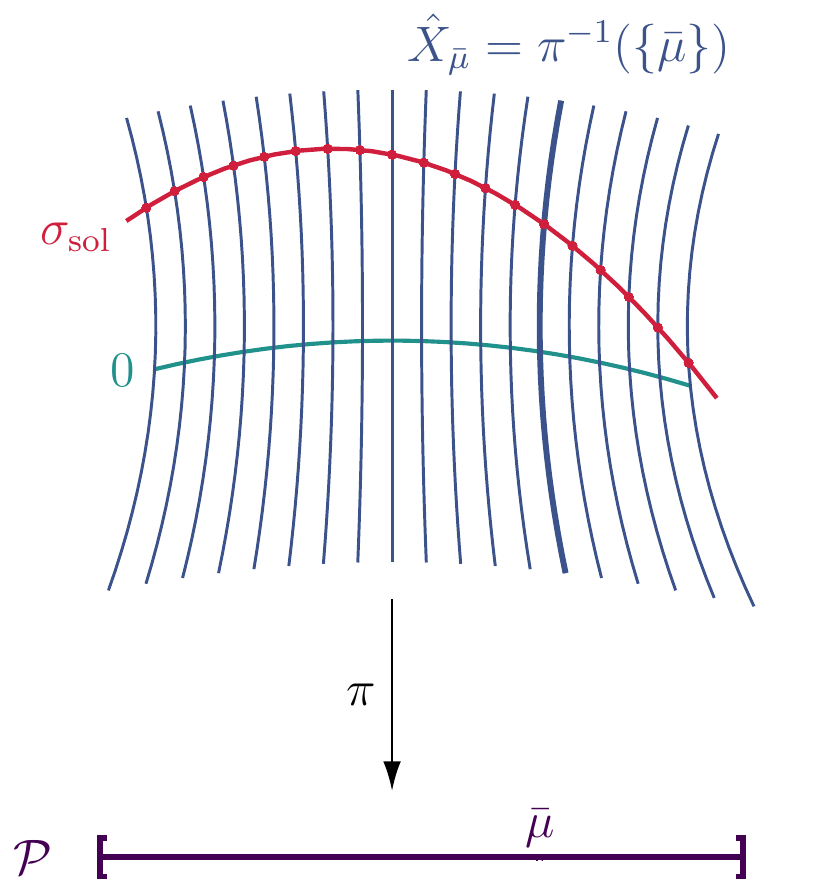}
	\caption{\label{fig:generalBundle}
		Visualization of a fiber bundle $(\totalSpace,\pi)$ over the base space $\Pcal$. Additionally, the \emph{solution section} $\solSection$ is visualized which takes one value (the solution of~\cref{eq:generalVariationalProblem}) in each fiber.}
\end{figure}

 We speak of a fiber/vector bundle if this collection allows for a smooth structure, i.e.\ if one can locally smoothly transition between the different fibers. While the notion of a vector bundle assumes finite-dimensional fibers, generalizations also allow for infinite-dimensional ones.
More precisely, we are interested in the case where each fiber $\trialSpace$ possesses a Banach or Hilbert-space structure, resulting in the definition of a  \emph{Banach-bundle}:
\newcommand{\refFiber}{X_{\mu_0}^{\textnormal{ref}}}
\begin{definition}[Banach bundle]\label{def:banachBundle}
  Let $\totalSpace$ be a topological space (called \emph{total space})and $\Pcal$ a Banach-manifold. Further, let
  $\pi\colon \totalSpace\to\Pcal$ be a surjective and continuous map. Denote by $\totalSpace_\mu\coloneqq\pi^{-1}(\{\mu\})$
  the fibers in $\totalSpace$ and assume that each fiber possesses the structure of a Banach space.
  $(\totalSpace,\pi)$ is called a \emph{Banach bundle} if it is \emph{locally trivial},
  i.e.\ for every base point $\mu_0\in\Pcal$ there
  exists an open neighborhood $\neighbourhood\subset\Pcal$, a Banach space $\refFiber$ and a homeomorphism
  $\tau\colon \pi^{-1}(\neighbourhood)\to \neighbourhood\times\refFiber$ such that:
  \begin{itemize}
  \item There holds $\projOne\circ\;\tau = \pi$ (with $\projOne(u,x)\coloneqq u$).
  \item For all $q\in \neighbourhood$ the Banach spaces $\totalSpace_q$ and $\refFiber$ are isomorphic via the
    restriction of $\tau$ to $\totalSpace_q$.
  \end{itemize}
\end{definition}

This definition formalizes the idea that a Banach bundle should locally ``look like'' the product $\neighbourhood\times\refFiber$ -- similar to manifolds locally ``looking like'' $\R^n$, see also~\cref{fig:rectification} for a visualization.
\begin{figure}
	\centering
	\includegraphics[width=0.5\linewidth]{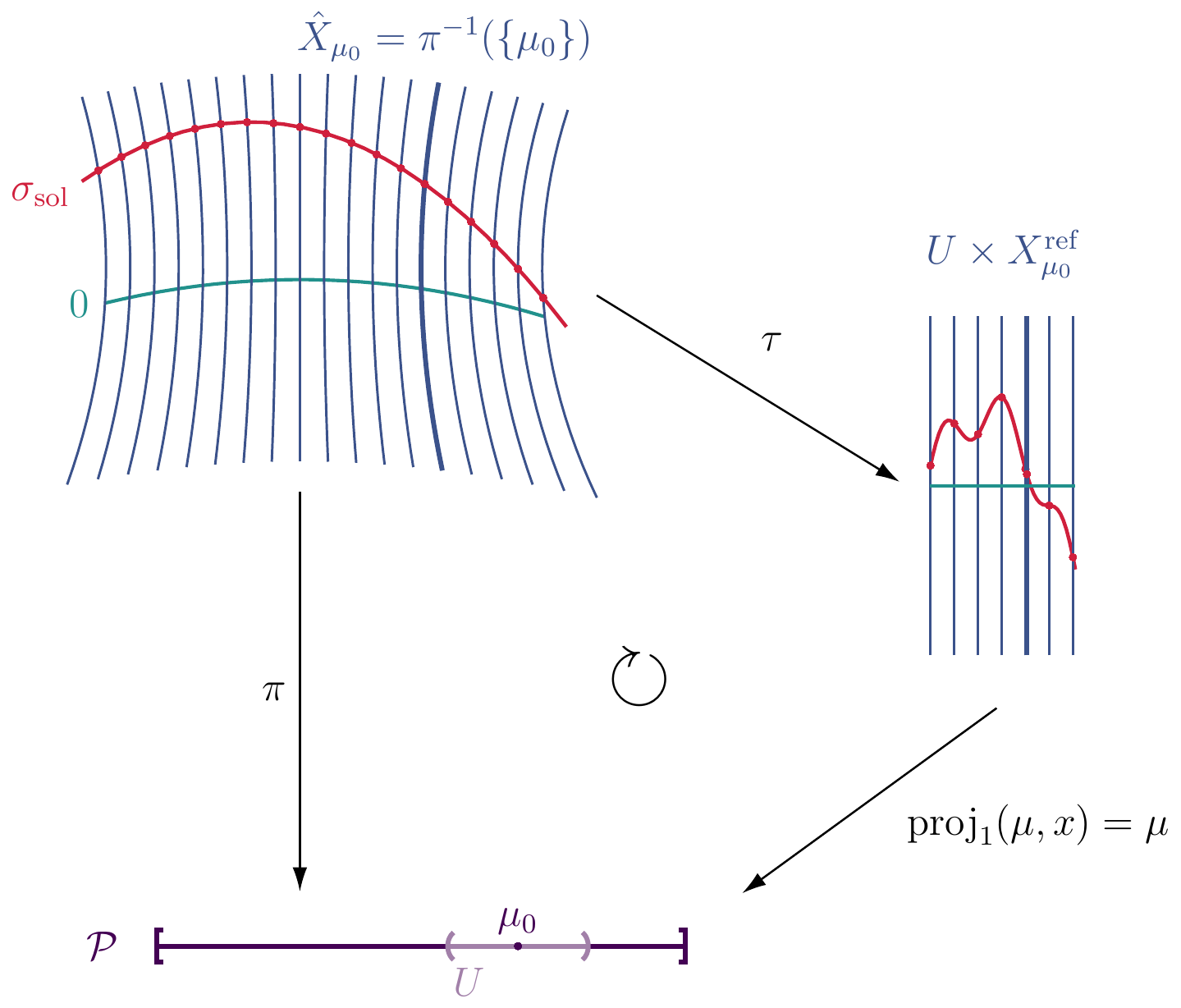}
	\caption{\label{fig:rectification}
		Visualization of the local trivialization $\tau$ in a neighborhood around the fiber $\totalSpace_{\mu_0}$. The diagram commutes.}
\end{figure}
\par
A second important concept which will be required are smooth maps into $\totalSpace$ taking exactly one value in each fiber, known as \emph{sections} of $\totalSpace$:
\begin{definition}[Cross section]\label{def:section}
	Let $\totalSpace$ be a fiber bundle over the base space $\Pcal$ with projection $\pi\colon\totalSpace\to\Pcal$. A \emph{(cross) section} of $\totalSpace$ is a continuous map
	\begin{equation*}
		\sigma\colon \Pcal\to\totalSpace, \quad \pi(\sigma(\mu)) = \mu.
	\end{equation*}
	In other words, for any $\mu\in\Pcal$ the function value $\sigma(\mu)$ is an element of the fiber $\totalSpace_\mu$.
\end{definition}
Note that~\cref{def:banachBundle} already starts with a topological space $\totalSpace$ and subsequently identifies the fibers $\totalSpace_\mu$ as the pre-image of the projection map $\pi\colon\totalSpace\to\Pcal$, i.e.\ as all the elements being mapped to the same base point $\mu\in\Pcal$. It is not obvious that our collection $\Xhat$ allows for a suitable topology such that the criteria in~\cref{def:banachBundle} hold.
Further, it is of interest whether such a topology would be unique or if multiple bundle structures could be defined.


The more general question of defining bundle structures based on a collection of fibers
has been investigated in the differential geometry community. We
present an idea going back to a theorem by J.Fell~\cite{fell1969extension}.
\begin{theorem}[Existence of a unique topology]\label{thm:bundles:fell}
  Let $\totalSpace$ be a set and $\pi\colon \totalSpace\to\Pcal$ a surjection such that each fiber
  $\totalSpace_\mu = \pi^{-1}(\{\mu\})$ is a Banach space. Let $\Gamma$ be a vector space of functions
  $\sigma\colon\Pcal\to\totalSpace$, $\sigma(\mu)\in\totalSpace_\mu$. Assume that
  \begin{enumerate}[label=(\roman*)]
  \item For every $\sigma_0\in\Gamma$, the map $\mu\mapsto \norm{\sigma_0(\mu)}_{\totalSpace_\mu}$ is continuous.
  \item For every $\mu_0\in\Pcal$, the set $\{\sigma(\mu_0) \;|\; \sigma\in\Gamma\}$ is dense in $\totalSpace_{\mu_0}$.
  \end{enumerate}
  Then $\totalSpace$ admits a unique topology such that $(\totalSpace,\pi)$ is a Banach bundle and 
  the (continuous) sections of $\totalSpace$ are precisely the elements in $\Gamma$.
\end{theorem}
\begin{remark}
	In the following we will call any set of functions $\Gamma$ fulfilling the assumptions in~\cref{thm:bundles:fell} an \emph{admissible set of sections}¸. Technically, these functions are not actual sections in the sense of~\cref{def:section} as there is no notion of continuity yet. However, as the unique topology from~\cref{thm:bundles:fell} automatically renders all of the elements in $\Gamma$ continuous, we are already referring to them as sections if it is clear that the topology induced by them is used.
\end{remark}

\begin{proof}
  The proof is simple and constructive. If the assumptions hold true, a topology on $\totalSpace$
  is generated by the collection
  \begin{equation*}
    \left\{ B_{\varepsilon}(\sigma,U) \;:\;
    \sigma\in\Gamma,\; U\subset\Pcal \;\textnormal{open}\right\},
  \end{equation*}
  where the $\varepsilon$-neighborhoods $B_{\varepsilon}(\sigma,U)$ are defined as
  \begin{equation*}
    B_{\varepsilon}(\sigma,U)
    \coloneqq \left\{\varphi\in\totalSpace \;:\; 
    \pi(\varphi)\in U,\;
    \norm{\varphi - \sigma(\pi(\varphi))}_{\totalSpace_{\pi(\varphi)}} < \varepsilon\right\}.
  \end{equation*}
  For the proof of uniqueness we refer to~\cite[Proposition~1.6]{fell1969extension}.
\end{proof}

Let us stress that the crucial ingredient for the preceding theorem is the existence of a number of sections
$\Gamma$ which then generate the topology on $\totalSpace$. While the theorem guarantees that for admissible sections $\Gamma$ there is exactly one corresponding topology/bundle structure, the choice of these generating sections is only restricted by the given criteria with possibly multiple feasible candidates.

\subsection{Linear approximation of the solution section}\label{sec:sectionalApprox}
We will now derive a suitable notion of approximability in the given setting. To that end, note that the set of solutions $\Mcal$ has itself the structure of a section. More precisely, we can view $\Mcal$ as the image of the \emph{solution section} 
\begin{equation}
	\solSection\colon\Pcal\to\Xhat, \qquad \mu\mapsto u_\mu
\end{equation}
assigning every $\mu\in\Pcal$ the corresponding solution $u_\mu$ to~\cref{eq:generalVariationalProblem}. Note again, that $\solSection$ is not necessarily smooth.
\par
Summarizing, we have seen that both the general differential geometric setting (i.e.\ the topology of $\totalSpace$), as well as the object we are actually interested in (the solution set $\Mcal$) are defined in terms of sections of $\totalSpace$ - it is thus only natural to introduce a corresponding notion of approximability. Based on~\cref{eq:kolmogorovNWidth}, one can define a \emph{Sectional Kolmogorov $N$-width} as follows:
\begin{definition}[Sectional Kolmogorov $N$-width]
	For a given possibly infinite-dimensional vector space of sections $\Gamma$ we define the Sectional Kolmogorov $N$-width of a given section $\sigma\colon\Pcal\to\totalSpace$ as 
	\begin{equation}\label{eq:sectionalNWidth}
		\kolmogorovSec{\sigma}
		\coloneqq
		\Inf_{\substack{\Gamma_N \subset \Gamma \;\textnormal{linear} \\ \dim(\Gamma_N)\,\leq\, N}}\;
		\sup_{\mu\in\Pcal}\; \Inf_{\sigma_N\in \Gamma_N} \norm{\sigma(\mu)-\sigma_N(\mu)}_{\totalSpace_\mu}.
	\end{equation}
\end{definition}
Note that this definition shifts the perspective from approximating the \emph{image} of the solution map by linear spaces to approximating the solution map itself by a linear combination of maps from a given set of available functions. 

\begin{remark}
	Some observations on the Sectional Kolmogorov $N$-width which directly follow from the definition:
	\begin{itemize}
		\item The Sectional Kolmogorov $N$-width is non-increasing in $N$.
		\item The Sectional Kolmogorov $N$-width is non-increasing in the set of sections with regard to vector subspace inclusion, i.e.\
		\begin{equation}
			\Gamma_1 \subseteq \Gamma_2
			\;\implies\;
			\kolmogorovSec[\Gamma_1]{\sigma} \geq \kolmogorovSec[\Gamma_2]{\sigma}.
		\end{equation}
		\item Let $\sigma$ be a (not necessarily continuous) sectional function and $\mu\mapsto\norm{\sigma(\mu)}_{X_\mu}$ be continuous. Then, the Sectional Kolmogorov $N$-width of $\sigma$ is finite as one has
		\begin{equation*}
			\kolmogorovSec{\sigma}
			\;\leq\; \sup_{\mu\in\Pcal}\norm{\sigma(\mu)}_{\trialSpace}
			\;<\; \infty
		\end{equation*}
		and $\Pcal$ is assumed to be compact.
	\end{itemize}
\end{remark}

Regarding the continuity of a given section $\sigma$ we can say the following:
\begin{remark}
	Let $\Gamma$ be an admissible set of sections and $\sigma_0\colon\Pcal\to\Xhat$ be an additional sectional function. Then, the following statements are equivalent:
	\begin{equation}
		\sigma_0\; \textnormal{is continuous}
		\;\iff\;
		\kolmogorovSec{\sigma_0} = 0 \quad\forallT N\geq 1.
	\end{equation}
\end{remark}
\begin{proof}
	By~\cref{thm:bundles:fell} we know that $\sigma_0$ is continuous if and only if it is contained in $\Gamma$. This immediately implies that the $N$-width vanishes as we can choose $\sigma_0$ to be contained in the approximating set $\Gamma_N$. Conversely, $\sigma_0$ must be an element of $\Gamma_N$ for the $N$-width to vanish which implies that it is also contained in $\Gamma$.
\end{proof}
Note, that for any admissible set $\Gamma$ and a sectional function $\sigma_0\not\in\Gamma$ we can consider the vector space
\begin{equation}
	\Gamma[\sigma_0] \coloneqq \{\sigma+\lambda\sigma_0 \;:\; \sigma\in\Gamma,\, \lambda\in\R\}
\end{equation}
and obtain $\kolmogorovSec[{\Gamma[\sigma_0]}]{\sigma_0}=0$. This illustrates that choosing the sections $\Gamma$ is not a question of \emph{optimal} approximation, it should rather be interpreted as the set of functions that we can evaluate easily for every given parameter (in particular this does \emph{not} include the solution section!). The more information is available on the concrete parameter-dependency, the better we can choose $\Gamma$. We will see examples for choices of $\Gamma$ throughout the remainder of this paper.

Following, it shall thus be discussed which statements are possible under more concrete assumptions on the dependency of the fibers on the parameter $\mu$.

\subsection{Parameter-independent fibers}\label{sec:parameterIndependent}
As a first sanity check, let us consider the special case where the function spaces/fibers do not actually dependent on the parameter. In this setting, the established manifold interpretation and the fiber bundle perspective are equivalent in the following way:

\begin{theorem}\label{thm:kolmogorovIdentity}
	Let all fibers $\trialSpace$ be parameter-independent, i.e.\ $\trialSpace = X$ and $\norm{\cdot}_{\trialSpace} = \norm{\cdot}_{X}$. For any function $\varphi\in X$ consider the associated constant section
	\begin{equation}\label{def:constantSection}
		\sigma_\varphi\colon\Pcal\to X,\qquad \sigma_\varphi(\mu)\equiv\varphi
	\end{equation}
	and the vector space $\Gconst \coloneqq \{ \sigma_\varphi  \;|\; \varphi\in X\}
		\;\cong\; X$
	which trivially satisfies the conditions in~\cref{thm:bundles:fell}.
	Then, the classic Kolmogorov $N$-width of $\Mcal=\textnormal{Im}(\solSection)$ and the Sectional Kolmogorov $N$-width of $\solSection$ coincide, i.e.\
	\begin{equation}
		\kolmogorov{\Mcal} \;=\; \kolmogorovSec[\Gconst]{\solSection}.
	\end{equation}
\end{theorem}
\begin{proof}
	This follows directly from the definitions as one has
	\begin{align*}
		\kolmogorov{\Mcal}
		&= \Inf_{\substack{X_N \subset X \;\textnormal{linear} \\ \dim(X_N)\,\leq\, N}}\; \sup_{u_\mu\in\Mcal} \Inf_{u_N\in X_N} \norm{u_\mu - u_N}_X \\
		&= \Inf_{\{\varphi_1,\ldots,\varphi_N\}\subset X}\;
		 \sup_{\mu\in\Pcal}\; \Inf_{u_N\in\linSpan\{\varphi_i\}} \norm{u_\mu-u_N}_X \\
		&= \Inf_{\{\varphi_1,\ldots,\varphi_N\}\subset X}\;
		 \sup_{\mu\in\Pcal}\; \Inf_{u_N\in\linSpan\{\varphi_i\}} \norm{u_\mu-\sigma_{u_N}(\mu)}_X \\
		&= \Inf_{\{\varphi_1,\ldots,\varphi_N\}\subset X}\; \sup_{\mu\in\Pcal}\; \Inf_{\sigma_N\in\linSpan\{\sigma_{\varphi_i}\}} \norm{\solSection(\mu)-\sigma_N(\mu)}_X \\
		&= \Inf_{\{\sigma_{\varphi_1},\ldots,\sigma_{\varphi_N}\}\subset\Gconst}\; \sup_{\mu\in\Pcal}\; \Inf_{\sigma_N\in\linSpan\{\sigma_{\varphi_i}\}} \norm{\solSection(\mu)-\sigma_N(\mu)}_X \\
		&=\kolmogorovSec[\Gconst]{\solSection}.
	\end{align*}
\end{proof}

\subsubsection{Relation to nonlinear Kolmogorov $N$-widths}
Let us remark that even for no parameter dependency it can be of interest to consider approximability with regard to other choices of sections $\Gamma$. In fact,
our definition then corresponds to nonlinear generalizations of the Kolmogorov $N$-width, see e.g.\ \cite{devore1989optimal,temlyakov1998nonlinear,devore2017foundation,rim2023performance}, which are in general formulated via an ``encoder-decoder'' pair $(E_N,D_N)$, $E_N\colon \Mcal\to\R^N$, $D_N\colon\R^N\to X$, a concrete example being the \emph{manifold width}
\begin{equation}\label{eq:nonlinearNWidth}
	\delta_N(\Mcal)
	\coloneqq
	\Inf_{E_N,D_N \,\textnormal{cont.}}\;
	\sup_{u_\mu\in\Mcal}\; \norm{u_\mu- (D_N\circ E_N)(u_\mu)}_{X}.
\end{equation}
Choosing a linear decoder $D_N$ and the encoder as the projection $E_N = \proj_{X_N}$, one recovers the classic Kolmogorov $N$-width~\cref{eq:kolmogorovNWidth}. The sectional $N$-width~\cref{eq:sectionalNWidth} only requires the decoder to be linear while the structure of the encoder amounts to
\begin{equation}
	E_N(u_\mu)_i \coloneqq (u_\mu,\sigma_i(\mu))_X
\end{equation}
which is then optimized over all choices of $\sigma_i$ from the admissible set $\Gamma$.
\par
Often, the main motivation for these nonlinear generalizations (and nonlinear methods realizing them in a quasi-optimal way) is a slow decay of the classic linear Kolmogorov $N$-width~\cref{eq:kolmogorovNWidth}. For instance, this occurs for equations involving discontinuities, transport phenomena or shocks~\cite{OhlbergerRave}. In the manifold-based view, one hopes for better approximability if instead of a linear space, nonlinear subspaces are used. While we do believe that this makes the need for nonlinear approximations clear, this perspective does not give any intuition on how different nonlinear approximations should compare.
\par
Let us first remark, that for methods based on nonlinear decoders, such as
quadratic manifolds~\cite{barnett2022quadratic} or neural network approximations~\cite{hesthaven2018nonintrusive}, the manifold perspective is what one should have in mind. These approaches are \emph{general purpose} in the sense that they do not inherently include problem-specific information. Note that we are talking about \emph{the design of the method}, the construction of concrete realizations (training) will of course use data such as snapshot solutions of the problem at hand.
In the case of neural networks, a connection to the section-based perspective can still be made:
\begin{example}
Consider an arbitrary artificial neural network (ANN) which has a final linear layer without bias, i.e.\
\begin{equation}
	u_\mu(x) \;\approx\; \textnormal{ANN}(x,\mu) = W^L y^{L-1}(x,\mu)
\end{equation}
In this case, we can interpret the sub-networks mapping to the value in the $i$-th node of the $(L-1)$-th layer as sections, i.e.\
\begin{equation*}
	\sigma_i(\mu) \coloneqq \left (y^{L-1}(\cdot,\mu)\right)_i.
\end{equation*}
This kind of network architectures can for example be found in the
random feature method, which has gained some interest in the
approximation of PDEs~\cite{chen2022bridging}.
\end{example}
The following pages will now show examples where we believe the section-based perspective possesses an advantage over the manifold-perspective:

\subsection{Regularity parametrization}\label{sec:regularityParam}
In the previous section, we utilized the notion of constant sections to show equivalence to the standard Kolmogorov $N$-width. Even if the fibers $\trialSpace$ actually vary with the parameter $\mu$, they might still have a non-empty intersection. In particular this is the case if the parameter influences the regularity of the elements in each fiber. In this case allowing us to identify the functions fulfilling the regularity requirements of all fibers. Formally, we continue focusing on the case where the spaces $\trialSpace$ fulfill the following assumption:
\vspace{0.5em}
\begin{enumerate}[label={(D1)}, align=left, leftmargin=*]
\item The intersection set $\Xinter\coloneqq \bigcap_{\mu\in\Pcal}\trialSpace$ is dense in
  $\trialSpace$ for all parameters $\mu\in\Pcal$.\label[assumption]{ass:trialSpaceDenseness}
\end{enumerate}
\vspace{0.5em}
We may then consider the constant sections
\renewcommand{\Gconst}{\Gamma_{\Xinter}}
\begin{equation*}
  \Gconst \coloneqq \{ \sigma_\varphi\colon\Pcal\to\Xhat,\, \;|\; \varphi\in\Xinter\}
  \;\cong\; \Xinter
\end{equation*}
which still fulfill the second assumption of~\cref{thm:bundles:fell}.

\begin{example}
	Let $\Omega\subset\R^d$ be a bounded domain and every space $X_\mu$ be of Sobolev type, i.e.\ $X_\mu\subseteq W^{k_\mu,p_\mu}(\Omega)$. Then, we have $\Cinf\cap C^0(\overline{\Omega})\subseteq\Xinter$ which is dense in every $X_\mu$.
\end{example}

It is in this generalized setting that we want to give an approximation theoretic result for the solution set of the variational problem~\cref{eq:generalVariationalProblem}. To that end, we state a similar assumption on the potentially also parameter-dependent test spaces $\testSpace$:
\vspace{0.5em}
\begin{enumerate}[label={(D2)}, align=left, leftmargin=*]
	\item The intersection set $\Yinter\coloneqq \bigcap_{\mu\in\Pcal}\testSpace$ is dense in
	$\testSpace$ for all parameters $\mu\in\Pcal$.\label[assumption]{ass:testSpaceDenseness}
\end{enumerate}
\vspace{0.5em}
A similar assumption was used in~\cite{renelt2023model} which was required to prove exponential approximability for parameter-dependent test spaces. The following theorem generalizes this statement to also allow for parametrized trial spaces:
\begin{theorem}[Approximability]\label{thm:FSMOR:kolmogorov}
  Let $\Mcal$ be defined as in~\cref{def:solutionSet} and let the following conditions hold:
  \begin{enumerate}
  \item The denseness assumptions~\cref{ass:trialSpaceDenseness,ass:testSpaceDenseness} hold.
  \item The intersection spaces $\Xinter$ and $\Yinter$ can be equipped with parameter-independent norms
    $\norm{\cdot}_{\Xinter}$, $\norm{\cdot}_{\Yinter}$ equivalent to the parameter-dependent norms
    $\norm{\cdot}_{\trialSpace}$, $\norm{\cdot}_{\testSpace}$,
    i.e.\ there exist $\mu$-independent equivalence
    constants $C,\tilde{C},c,\tilde{c} > 0$ such that
    \begin{align}
      &\tilde{c}\norm{u}_{\Xinter} \;\leq\; \norm{u}_{\trialSpace} \;\leq\; \tilde{C}\norm{u}_{\Xinter}
      \qquad\forallT u\in\Xinter,\label{eq:FSMOR:normEquivalenceTrial} \\
      &c\norm{v}_{\Yinter} \;\leq\; \norm{v}_{\testSpace} \;\leq\; C\norm{v}_{\Yinter}
      \qquad\forallT v\in\Yinter.\label{eq:FSMOR:normEquivalenceTest}
    \end{align}
  \item The $b_\mu$ and $f_\mu$ are parameter-separable on $\Xinter$, $\Yinter$,
    i.e.\ there exist continuous functions
    $\theta^b_i, \theta^f_i\colon \Pcal \rightarrow \R$ and (bi-) linear and continuous functionals
    $b_i\colon\Xinter\times\Yinter \rightarrow \R$, $f_i\colon \Yinter\rightarrow\R$ such that
    for all $u\in\Xinter$ and $v\in\Yinter$ we have the representation
  \begin{equation*}
    b_\mu(u,v) \;=\; \sum_{i=1}^{Q_b}\theta^b_i(\mu)\,b_i(u,v), \qquad
    f_\mu(v) \;=\; \sum_{i=1}^{Q_f}\theta^f_i(\mu)\,f_i(v).
  \end{equation*}
  \end{enumerate}
  Then, the Sectional Kolmogorov $N$-width $\kolmogorovSec[\Gconst]{\solSection}$ of the solution manifold $\Mcal$
  decays exponentially, i.e.\
  \begin{equation*}
    \kolmogorovSec[\Gconst]{\solSection}
    \;\leq\;
    \alpha\cdot \exp(-\beta N^{1/Q_b}).
  \end{equation*}
\end{theorem}
\begin{proof}
  The assumptions imply that the parameter-independent spaces
  \begin{equation*}
    \xcalBarZ \;\coloneqq\; \clos_{\norm{\cdot}_{\Xinter}}(\Xinter), \qquad
    \ycalBarZ \;\coloneqq\; \clos_{\norm{\cdot}_{\Yinter}}(\Yinter).
  \end{equation*}
  are isomorphic to every $\trialSpace$ or $\testSpace$, respectively.
  Therefore, the problem 
  \begin{equation*}
    \FindT u_\mu\in\trialSpace\colon\quad
    b_\mu(u_\mu, v) \eq f_\mu(v)
    \qquad \forallT v\in\ycalBarZ.
  \end{equation*}
  is equivalent to~\cref{eq:generalVariationalProblem} in the sense that it results in the same set of solutions $\Mcal$. If we now consider
  \begin{equation}\label{eq:FSMOR:nonparametricProblem}
    \FindT u_\mu\in\xcalBarZ\colon\quad
    b_\mu(u_\mu, v) \eq f_\mu(v)
    \qquad \forallT v\in\ycalBarZ.
  \end{equation}
  and the associated solution set $\overline{\Mcal}\subseteq\xcalBarZ$, 
  we can apply the well-established results for inf-sup stable problems (see e.g.\ \cite{urban2023}) as
  problem \cref{eq:FSMOR:nonparametricProblem} no longer involves parameter-dependent spaces. Therefore,
  $\kolmogorov{\overline{\Mcal}}$ decays with the proposed rate.
  \par
  Using arguments similar to ~\cref{thm:kolmogorovIdentity} and exploiting the norm equivalence we see that
  \begin{align*}
	\kolmogorovSec[\Gconst]{\solSection}
	&= \Inf_{\substack{U_N \subset \Xinter \;\textnormal{linear} \\ \dim(U_N)\,\leq\, N}}\; \sup_{u_\mu\in\Mcal} \Inf_{u_N\in U_N} \norm{u_\mu - u_N}_{\trialSpace} \\
 	&\lesssim \Inf_{\substack{U_N \subset \Xinter \;\textnormal{linear} \\ \dim(U_N)\,\leq\, N}}\; \sup_{u_\mu\in\Mcal} \Inf_{u_N\in U_N} \norm{u_\mu - u_N}_{\Xinter} \\
 	&\leq \Inf_{\substack{U_N \subset \xcalBarZ \;\textnormal{linear} \\ \dim(U_N)\,\leq\, N}}\; \sup_{u_\mu\in\Mcal} \Inf_{u_N\in U_N} \norm{u_\mu - u_N}_{\Xinter} \\
 	&= \kolmogorov{\overline{\Mcal}}.
	\end{align*}
\end{proof}

\begin{remark}
	Note that the distinguishing criterion in the preceding theorem is the norm equivalence which ultimately implies exponential approximability. In~\cref{sec:friedrichs} we will use this criterion to classify different Friedrichs' systems based solely on inspecting the parameter-dependent norm on the trial spaces $\trialSpace$.
\end{remark}

\subsection{Geometrical and general transformation-based parametrizations}\label{sec:geometricalParametrization}
In the following we want to briefly discuss parametrizations that do not fulfill~\cref{ass:trialSpaceDenseness}. In these cases it is not possible to identify enough constant sections to make the set $\Gconst$ admissible -- more structural information is required.
\par
Following, we want to briefly recall nonlinear reduction strategies involving parameter-dependent transformations - either of the underlying domain or of the solutions themselves. We believe that here our framework can provide additional intuition since, as it will become clear, the information about the transformation is retained. In contrast, if one solely thinks about the approach generating \emph{some} nonlinear structure, any further information on its structure is lost.
In the following, we further detail how exactly some of these methods fit into the fiber-based framework and how the corresponding set of sections $\Gamma$ may be interpreted.

\subsubsection{Geometrical parametrizations}
An important class of problems are those involving parametrized geometrical information, i.e.\ where a parameter-dependency of the form $\trialSpace = X(\Omega_\mu)$ occurs.
Let us focus on methods that may be characterized as \emph{Lagrangian} approaches. All of them involve a transformation
\begin{equation*}
	T_\mu\colon\Omega_0\to\Omega_\mu
\end{equation*}
that maps from a reference configuration $\Omega_0$ to the physical domain $\Omega_\mu = T_\mu(\Omega_0)$ (where $\Omega_\mu$ can also be identical for all parameters).
Once again interpreting the set of sections $\Gamma$ as an ``accessible identification'' of functions from different fibers, an intuitive choice for $\Gamma$ are the non-constant sections
\begin{equation}\label{eq:geomSections}
	\sigma_{\varphi,\textnormal{geom}}(\mu) = \varphi(T^{-1}_\mu(\cdot))
\end{equation}
for a suitable set of functions $\varphi$ defined on the reference domain $\Omega_0$. Of course further restrictions on the transformation $T_\mu$ are usually required.
Let us give a few examples of methods falling into this class:

\noindent\textbf{Parametrized geometries}\\
The most immediate problem class are equations posed on a domain $\Omega_\mu$ which actually varies with the parameter requiring the use of function spaces $X_\mu=X(\Omega_\mu)$ on the different domains.
Examples discussed in the MOR community are mostly shape optimization problems, see e.g.\ \cite{thiyagarajan2005shapeOptimizationRB,lassila2010parametric,manzoni2012shape}.

\noindent\textbf{Domain transformation / registration approaches}\\
Even if the physical domain $\Omega_\mu$ remains identical for all solutions, it can still be advantageous to consider parameter-dependent transformations of the domain in order to align features of the solutions appearing at different locations. This idea was initially presented in~\cite{welper2017interpolation} which proposed an ansatz of the form
\begin{equation} 
	u_\mu(x) \approx \sum_{i=1}^M c_\mu^i\, u_i(\phi_\mu^i(x)), \qquad \phi_\mu^i\colon\Omega\to\Omega.
\end{equation}
This again fits our setting choosing the sections
\begin{equation}
	\sigma_{\varphi,i}(\mu) \coloneqq \varphi(\phi^i_\mu(\cdot)), \quad\textnormal{for}\;\varphi\in X,\quad i=1,\ldots,M
\end{equation}
A general approach to performing the transformation was presented in~\cite{taddei2020registration}.

\noindent\textbf{Diffeomorphic space-time transformations}\\
For instationary problems, one may even consider a transformation of the whole (unparametrized) space-time domain $\Omega_T=\Omega\times[0,T]$. In~\cite{kleikamp2022nonlinear} the ansatz
\begin{equation}\label{eq:diffeomorphicTransformations}
	u_\mu(x,t) \approx u_0(\phi_\mu^{-1}(x,t)), \qquad \phi_\mu\colon\Omega_T\to\Omega_T
\end{equation}
with a single given reference snapshot $u_0$ and parameter-dependent diffeomorphisms $\phi_\mu$ is made. 

\noindent\textbf{Mesh-transforming methods}\\
Instead of considering the transformation of the domain $\Omega$ independent of the discretization, one can also directly consider transformations of the discrete mesh $\Omega_h$ \cite{salmoiraghi2018}. Then, we may interpret these methods as discrete problems seeking solutions in the parametrized function spaces
\begin{equation}
	X_\mu = X(\Omega_h^\mu) = X(T_\mu(\Omega_h)).
\end{equation}
Methods based on this idea include for example the implicit feature tracking approach~\cite{mirhoseini2023model}.

\noindent\textbf{The Arbitrary Lagrangian Eulerian (ALE) framework}\\
In~\cite{torlo2020ALE,nonino2024calibration} a generalized transformation ansatz is proposed, introduced as the \emph{arbitrary Lagrangian-Eulerian framework}. Still based on the idea of aligning solution features and thus facilitating better approximability in the reference configuration, the transformation $T_\mu$ is generalized to
\begin{equation*}
	T_\theta\colon\Omega_0\to\Omega
\end{equation*}
combined with a \emph{calibration map}
\begin{equation}
	\theta\colon[0,T]\times\Pcal\to\Pcal_{geom}
\end{equation}
determining the geometric parameters $\theta(t,\mu)$ given the current time $t\in[0,T]$ and physical parameter $\mu\in\Pcal$. 

\subsubsection{Nonlinear solution transformations}
Compared to the previous approaches, the following methods assume that instead of the domain, the solution itself is transformed by a nonlinear mapping \begin{equation}\label{eq:nonlinearSolTransform}
	\psi_\mu\colon X_0\to\trialSpace 
\end{equation}
which lets us characterize them as \emph{Eulerian approaches}. Similar to the geometrical framework, the spaces $\trialSpace$ may also coincide.
In comparison to~\cref{eq:geomSections}, we now identify sections of the form
\begin{equation}\label{eq:sectionsSolutionTransformation}
	\sigma_\varphi(\mu) \coloneqq \psi_\mu(\varphi(\cdot)), \qquad \varphi\in X_0.
\end{equation}
To make this compatible with our framework, we thus have to define the vector space of sections $\Gamma$ with the operations
\begin{align}
	\sigma_{\varphi_1} \oplus \sigma_{\varphi_2} &\coloneqq \sigma_{\varphi_1+\varphi_2},\\
	\lambda\odot\sigma_{\varphi} &\coloneqq \sigma_{\lambda\varphi}.
\end{align}
In particular, note that due to the nonlinearity this means that for a given parameter $\mu\in\Pcal$, one generally has
\begin{equation*}
	(\sigma_{\varphi_1} \oplus \sigma_{\varphi_2})(\mu) \;\neq\; \sigma_{\varphi_1}(\mu) + \sigma_{\varphi_2}(\mu)
\end{equation*}
where on the right hand side $+$ is the standard addition in $\trialSpace$.

\noindent\textbf{Diffeomorphic transformations}\\
Even though it is based on a domain transformation, the approach~\cref{eq:diffeomorphicTransformations} is actually rather Eulerian. This is because the authors parametrize the diffeomorphism $\phi_\mu$ as an element of a geodesic in the diffeomorphism group, characterized as
\begin{equation}
	\phi_\mu = \exp(\vec{v}_\mu)(1)
\end{equation}
with $\exp$ denoting the exponential map and an initial vector field $\vec{v}_\mu\colon\Omega\to\R^d$. The corresponding sections are thus rather given by
\begin{equation}
	\sigma_{\vec{v}}(\mu) \coloneqq u_0((\exp(\vec{v})(1)).
\end{equation}

\noindent\textbf{Shifted POD}\\
The shifted POD~\cite{burela2023spod} follows a similar idea, assuming that the solution $u_\mu$ is given by a superposition of differently shifted fields, i.e.
\begin{equation}
	u_\mu(x,t) \approx \sum_{k=1}^{n_k} T^k(u^k_\mu(x,t))
\end{equation}
where the operator $T^k$ acts as a (time-dependent) shift
\begin{equation}
	T^k(u(x,t)) = u(x-\Delta^k_\mu(t), t).
\end{equation}

\noindent\textbf{Method of freezing}\\
The method of freezing~\cite{ohlberger2013nonlinear} considers instationary problems of the form
\begin{equation}
	\partial_t u_\mu(x,t) + A_\mu(u_\mu) = 0
\end{equation}
and assumes a solution structure
\begin{equation}
	u_\mu(x,t) \approx g_\mu(x,t)\odot v_\mu(x,t)
\end{equation}
where $\odot$ denotes the action of an element $g_\mu$ from a Lie group $G$ on a function $v_\mu\in V$. We may interpret this as a fibration of the form 
$\trialSpace = g_\mu\odot V$ with an admissible set $\Gamma_G$ given by the sections
\begin{equation}
	\sigma_v(\mu) \coloneqq g_\mu\odot v, \qquad v\in V.
\end{equation}
In a certain sense, both Shifted POD and the diffeomorphic mappings are examples, the first considering the translation group and the second the diffeomorphism group as the underlying Lie-group $G$.

\section{Application to Friedrichs' systems}\label{sec:friedrichs}
Given the abstract results presented in the previous section, we will now turn our attention to a specific class of PDE problems, namely Friedrichs' systems. As it will become evident, parametrized problems of this type naturally involve parameter-dependent function spaces are thus a meaningful application for the concepts presented in the preceding section. After formally introducing the problem class, we apply our results and show under which additional constraints the requirements for an exponential decay of the (sectional) $N$-width are satisfied.

\subsection{Basic theory of Friedrichs' systems}\label{sec:theoryFriedrichs}
The Friedrichs' framework generalizes a large class of linear first-order PDE-operators into a
single abstract setting and was introduced by Friedrichs in~\cite{friedrichs1958}.
While the initial theory was based on the notion of strong differentiability,
the ideas have since been extended to the modern concepts
of weak and ultraweak solutions originating from Sobolev
spaces~\cite{ernGuermond2007,ernGuermond2006friedrichs1,
  ernGuermond2006friedrichs2,ernGuermond2008friedrichs3}.
In this section we provide a brief introduction to the theory and state the weak and ultraweak variational formulation. For details and proofs we refer the reader to~\cite{ernGuermond2007}.

\begin{definition}[Friedrichs' operator]
	A (parametrized) Friedrichs' operator is a vector-valued
	differential operator $\FSop$ of the form
	\begin{equation}\label{eq:FS:friedrichsOperator}
	  \FSop\colon \CinfM\to\LtwoM,\qquad
	  \FSop u \;=\; \FScoeff[0] u \;+\; \sum_{i=1}^{d} \FScoeff \frac{\partial u}{\partial x_i}
	\end{equation}
	with matrix-valued coefficient functions $\FScoeff$ satisfying
	\begin{equation*}
	  \FScoeff \in [L^\infty(\Omega)]^{m\times m}, \qquad
	  \nabla\cdot\FSop \coloneqq\;
	  \sum_{i=1}^d \frac{\partial \FScoeff}{\partial x_i}
	  \;\in [L^\infty(\Omega)]^{m\times m}.
	\end{equation*}
	Additionally, the following two properties need to be satisfied:
	\vspace{0.5em}
	\begin{enumerate}[label={(FS\arabic*)}, align=left, leftmargin=*]\itemsep0.5em
	\item $\FScoeff \; = \left(\FScoeff\right)^T$
	  \quad for all $i=1,\dots,d$,\label[assumption]{ass:FS:symmetry}
	\item $\FScoeff[0] + \left(\FScoeff[0]\right)^T - \nabla\cdot\FSop \;>\; 2 \varepsilon I_m$
	  \quad for some $\varepsilon  > 0$.\label[assumption]{ass:FS:positivity}
	\end{enumerate}
	\vspace{0.5em}
	It shall further be assumed that the parameter set~$\Pcal$ is a compact set in~$\R^p$ and that the mappings $\mu\mapsto\FScoeff$ are continuous for all $i=0,\ldots,d$.
\end{definition}
It is worth noting that for non-scalar systems~\cref{ass:FS:positivity} can under certain conditions
be relaxed, see e.g.~\cite{ernGuermond2008friedrichs3} for details.

\begin{definition}[Graph-space]
	The graph space $\graphSpace$ is defined as the space of all square-integrable functions which possess a weak $\FSop$-derivative, i.e.\
	\begin{equation}\label{eq:FS:graphSpace}
		\graphSpace \;\coloneqq\; \{ u\in\LtwoM \;|\; \FSop u\in\LtwoM\}.
	\end{equation}
	A norm on $\graphSpace$ is then given by the graph-norm
	\begin{equation}\label{eq:FS:graphNorm}
		\norm{u}_{\graphSpace}^2 \;\coloneqq\; \norm{u}_{\LtwoM}^2 + \norm{\FSop u}_{\LtwoM}^2.
	\end{equation}
\end{definition}
One immediately verifies that the inclusion $H^1(\Omega)^m \subseteq \graphSpace \subseteq \LtwoM$ holds for any Friedrichs' operator $\FSop$. Further, we can define the formal adjoint operator corresponding to $A$ as
\begin{equation}\label{eq:FS:adjointOperator}
  \FSadjOp\colon \CinfM\to\LtwoM,\qquad
  \FSadjOp v \;=\; (\left(\FScoeff[0]\right)^T - \nabla\cdot\FSop)v - \sum_{i=1}^{d} \left(\FScoeff\right)^T \frac{\partial v}{\partial x_i}
\end{equation}
and check that $\FSadjOp$ is itself a Friedrichs' operator. Moreover, one directly verifies that the
corresponding graph-space~$\graphSpaceAdj$ is isomorphic to the primal graph-space~$\graphSpace$.
\par
In order to derive a well-posed variational problem, additional boundary conditions need to be imposed. Following~\cite{ernGuermond2007}, we define the \emph{boundary operator} $\FSboundaryD\colon \graphSpace\to\graphSpace'$ by
\begin{equation}\label{eq:FS:boundaryOperator}
  (\FSboundaryD u)(v) \;\coloneqq\; (\FSop u,v)_{\LtwoM} - (u,\FSadjOp v)_{\LtwoM} \qquad \forallT u,v\in\graphSpace.
\end{equation}
In particular, this operator vanishes for compactly supported functions~$u\in\CcinfM$ which justifies the term boundary operator.
Additionally, one can show that $\FSboundaryD$ is self-adjoint \cite[Lemma~2.3]{ernGuermond2007}.
For coefficients $\FScoeff$ sufficiently smooth up to the boundary (e.g.\ $\FScoeff\in C^0(\overline{\Omega})$)
one further has the representation
\begin{equation}\label{eq:FS:representationBoundaryOpD}
  (\FSboundaryD u)(v) \;=\; \int_{\partial\Omega} v^T \underline{D}_\mu u \diff s,
  \qquad \underline{D}_\mu\coloneqq \sum_{i=1}^d n_i \FScoeff
\end{equation}
where $\vec{n} = (n_1,\dots,n_d)$ denotes the unit outer normal to the boundary $\partial\Omega$. Further characterizations of the regularity of $\underline{D}_\mu$ can be found e.g.\ in~\cite{rauch1985symmetric}.
\par
To prescribe boundary conditions, the operator $\FSboundaryD$ is then paired with a second,
potentially non-unique \emph{admissible boundary operator}
$\FSboundaryM\colon\graphSpace\to\graphSpace'$ which needs to satisfy the conditions
\vspace{0.5em}
\begin{enumerate}[label={(M\arabic*)}, align=left, leftmargin=*]\itemsep0.5em
\item $(\FSboundaryM u)(u) \;\geq\; 0$
  \quad for all $u\in\graphSpace$,\label[assumption]{ass:FS:boundaryM:1}
\item $\graphSpace = \ker(\FSboundaryD-\FSboundaryM) \;+\; \ker(\FSboundaryD+\FSboundaryM)$.\label[assumption]{ass:FS:boundaryM:2}
\end{enumerate}
\vspace{0.5em}
Given such an operator $\FSboundaryM$ we can define the closed subspaces
\begin{align*}
  \graphSubspace &\coloneqq\; \ker(\FSboundaryD-\FSboundaryM) \;\subset\; \graphSpace, \\
  \graphSubspaceAdj &\coloneqq\; \ker(\FSboundaryD+\FSboundaryM^*) \;\subset\; \graphSpaceAdj
\end{align*}
and show that the restriction to these subspaces implies coercivity of $\FSop$ and $\FSadjOp$ in the following
sense:
\begin{proposition}[$L^2$-coercivity]\label{prop:FS:l2coercivity}
  The restricted Friedrichs' operators
  \begin{equation*}
    \FSop\colon\graphSubspace\to\LtwoM
    \quad\andT\quad
    \FSadjOp\colon\graphSubspaceAdj\to\LtwoM
  \end{equation*}
  are coercive, i.e.\
  \begin{equation*}
    (\FSop u,u)_{\LtwoM} \;\geq\; \varepsilon\norm{u}_{\LtwoM}^2
    \quad\andT\quad
    (v,\FSadjOp v)_{\LtwoM} \;\geq\; \varepsilon\norm{v}_{\LtwoM}^2.
  \end{equation*}
\end{proposition}
This is a crucial ingredient to prove that
$\FSop\colon\graphSubspace\to\LtwoM$ constitutes an isomorphism. Equivalently, the following theorem holds:
\begin{theorem}[Well-posedness of the weak problem~{\cite[Thm.~2.5]{ernGuermond2007}}]\label{thm:FS:weakProblem}
For any $f_\mu\in\LtwoM$, the problem
  \begin{equation}\label{eq:FS:weakProblem}
    \FindT u_\mu\in \paramGraphSubspace\colon
    \quad (\FSop u_\mu,v)_{\LtwoM} \eq f_\mu(v) \qquad \forallT v\in\LtwoM.
  \end{equation}
  is well-posed.
\end{theorem}



We will later also consider the following ultraweak formulation:
\begin{theorem}[Well-posedness of the ultraweak problem]\label{thm:FS:ultraweakProblem}
  The ultraweak problem
  \begin{equation}\label{eq:FS:ultraweakProblem}
    \FindT u_\mu\in\LtwoM\colon\quad
    (u_\mu,\FSadjOp v)_{\LtwoM} \eq f_\mu(v) \qquad\forallT v\in\paramGraphSubspaceAdj.
  \end{equation}
  is well-posed for any right-hand side $f_\mu\in\graphSubspaceAdj'$.
\end{theorem}
\begin{proof}
	As the ultraweak formulation is not explicitly discussed in~\cite{ernGuermond2007}, we give a short proof. Following from the \BNB~ theorem, we need to show continuity and inf-sup-stability. The continuity follows directly from Cauchy-Schwarz as the test space $\graphSubspaceAdj$ is equipped with the graph norm.
	\par
	 To show the stability, we recall that by~\cite[Thm.~2.5]{ernGuermond2006friedrichs1} the adjoint operator $\FSadjOp\colon \graphSubspaceAdj\to\LtwoM$ constitutes an isomorphism. We may thus estimate
  \begin{align*}
    &\Inf_{u\in\LtwoM} \sup_{v\in\graphSubspaceAdj}
      \frac{|(u,\FSadjOp v)_{\LtwoM}|}{\norm{u}_{\LtwoM}\norm{v}_{\graphSpaceAdj}} \\
    \geq\; &\inf_{u\in\LtwoM}
             \frac{|(u,\FSadjOp \FSop^{-*}u)_{\LtwoM}|}{\norm{u}_{\LtwoM}\norm{\FSop^{-*}u}_{\graphSpaceAdj}} \\
    =\; &\Inf_{u\in\LtwoM}
          \frac{\norm{u}_{\LtwoM}}{(\norm{\FSop^{-*}u}_{\LtwoM}^2 + \norm{u}_{\LtwoM}^2)^{1/2}} \\
    \geq\; &(1 + \opNorm{\FSop^{-*}}^2)^{-1/2}
  \end{align*}
  where the last expression is bounded away from zero due to the bounded inverse theorem.
\end{proof}
%

\subsection{Approximation results for Friedrichs' systems}\label{sec:approximationFriedrichs}
In the following, we investigate the approximability of the solution sets corresponding to the weak formulation~\cref{eq:FS:weakProblem}. The ultraweak problem~\cref{eq:FS:ultraweakProblem} was already partially investigated in~\cite{renelt2023model} and can be considered similarly.
According to~\cref{thm:FSMOR:kolmogorov}, we can base any further characterization entirely on the structure of the parameter-dependent spaces $\graphSubspace$ and $\graphSubspaceAdj$. We further notice that already scalar Friedrichs' systems, which amount to linear advection-reaction equations, cover three distinct possibilities for parametrized trial spaces:

\begin{example}[Scalar Friedrichs' systems]
	Let us consider the scalar Friedrichs' operator $A_\mu u = \nabla\cdot(\vec{b}_\mu u) + c_\mu u$ on $\Omega=[0,1]^2$. Defining the inflow boundary
	\begin{equation*}
	\GinMu\coloneqq \{x\in\partial\Omega \;:\; \vec{b}_\mu(x)\cdot\vec{n}(x) < 0\}
	\end{equation*}
	one can show that the only admissible boundary operator is $(\FSboundaryD-\FSboundaryM)(u) = u\restr{\GinMu}$ and we thus have
	\begin{equation}
		\paramGraphSubspace \;\cong\; \left\{u\in\Ltwo \;:\; \vec{b}_\mu\nabla u\in\Ltwo^2,\; u=0 \onT\GinMu\right\}
	\end{equation}
\begin{enumerate}
	\item First, consider the case where only the reaction coefficient $c_\mu$ is parameter-dependent. Then, the corresponding graph spaces are identical as sets but are still equipped with different $\mu$-dependent norms. 
	\item Now, let $\vec{b}_\mu = (\cos(\mu),\sin(\mu))^T$ with $\mu\in[\varepsilon, \tfrac{\pi}{2}-\varepsilon]$ for some small $\varepsilon>0$. Then, the inflow boundary $\Gin=\{0\}\times[0,1]\cup[0,1]\times\{0\}$ is identical for all $\mu$ and one verifies that in this case~\cref{ass:trialSpaceDenseness} holds.
	\item Finally, let again $\Omega=[0,1]^2$ and $\vec{b}(\mu) = (\cos(\mu),\sin(\mu))^T$ but with angle $\mu\in[0,2\pi]$. In this case, one verifies that the intersection set $\Xinter$ is actually $H_0^1(\Omega)$ which is not dense in any $\paramGraphSubspace$, i.e.\ \cref{ass:trialSpaceDenseness} does not hold. 
\end{enumerate} 
\end{example}
We can generalize the observations from this example in the following way:

\begin{lemma}[{Sufficient condition for~\cref{ass:trialSpaceDenseness,ass:testSpaceDenseness}}]\label{lemma:friedrichsDenseness}
	Let the boundary operators $D-M$ and $D+M^*$ be parameter-independent. Then,~\cref{ass:trialSpaceDenseness,ass:testSpaceDenseness} hold.
\end{lemma}
\begin{proof}
	This is evident as, by a Meyers-Serrin-type argument, the space $\paramGraphSubspace$ can be identified as the completion of the smooth functions
	\begin{equation*}
		C^\infty_{\FSboundaryD-\FSboundaryM}(\Omega)^m \coloneqq\; \left\{\varphi\in\CinfM \;:\; (\FSboundaryD-\FSboundaryM¸)(\varphi)=0\right\}
	\end{equation*}
	under the parameter-dependent graph-norm $\norm{\cdot}_{\paramGraphSpace}$. If $D-M$ is parameter-independent, one has
	$C^\infty_{D-M}(\Omega)^m \subseteq\Xinter$ and thus~\cref{ass:trialSpaceDenseness} holds. The statement for~\cref{ass:testSpaceDenseness} follows analogously.
\end{proof}

\begin{lemma}[Sufficient condition for a $\mu$-independent norm]\label{lemma:friedrichsNorm}
  Let $A_\mu$ be a Friedrichs' operator which additionally fulfills
  \vspace{0.5em}
  \begin{enumerate}[label={(N1)}, align=left, leftmargin=*]
    \item $A^i_\mu = \ahat\,\tilde{A}^i$
	\quad for all $i=1,\ldots,d$, where $\ahat\in\Linfty$, $\ahat\geq\kappa>0$ a.e. \label[assumption]{ass:zerothOrder}
  \end{enumerate}
  \vspace{0.5em}
  Then for both the weak and ultraweak variational
  formulations~\cref{eq:FS:weakProblem,eq:FS:ultraweakProblem}, the second assumption in~\cref{thm:FSMOR:kolmogorov} holds. The parameter-independent
  norm is given as
  \begin{equation}\label{eq:FSMOR:zerothOrderFSNorm}
    \norm{u}_0^2 \;\coloneqq\; \norm{u}_{\LtwoM}^2 + \norm{\sum_{i=1}^d \tilde{A}^i \partial_{x_i}u}_{\LtwoM}^2.
  \end{equation}
  for $\norm{\cdot}_{\Xinter} = \norm{\cdot}_0$ or $\norm{\cdot}_{\Yinter} = \norm{\cdot}_0$, respectively.
\end{lemma}
\begin{proof}
  Let us first consider the weak formulation.
  The first bound is obtained by a simple application of the triangle inequality and Young's theorem:
  \newcommand{\FSFOterm}{\sum_{i=1}^d \tilde{A}^i \partial_{x_i}}
  \begin{align*}
    \norm{u}_{\paramGraphSpace}^2
    &= \norm{u}_{\LtwoM}^2 + \norm{A_\mu u}_{\LtwoM}^2 \\
    &\leq \norm{u}_{\LtwoM}^2 + \left(\norm{A_\mu^0 u}_{\LtwoM}
      + \norm{\ahat\FSFOterm u}_{\LtwoM} \right)^2 \\
    &\leq (1+2\norm{A_\mu^0}_{\Linfty}^2)\norm{u}_{\LtwoM}^2
      + 2\norm{\ahat}_{\Linfty}^2\norm{\FSFOterm u}_{\LtwoM}^2 \\
    &\leq \max\{2\norm{\ahat}_{\Linfty}^2,\, 1+2\norm{A_\mu^0}_{\Linfty}^2\} \norm{u}_0^2. \\
  \end{align*}
  For the lower bound we perform a similar estimation by
  expanding the first order term and again using Young's theorem:
  \begin{align*}
    \norm{u}_0^2 \;
    &\leq \norm{u}_{\LtwoM}^2
	+ \norm{1/\ahat}_{\Linfty}^2\norm{\ahat\FSFOterm u}_{\LtwoM}^2 \\
    &\leq \norm{u}_{\LtwoM}^2
      + \kappa^{-2}\left(\norm{\ahat\FSFOterm u + A_\mu^0 u}_{\LtwoM}
      + \norm{A_\mu^0 u}_{\LtwoM}\right)^2 \\
    &\leq \left(1+´2\kappa^{-2}\norm{A_\mu^0}_{\Linfty^m}^2\right)\norm{u}_{\LtwoM}^2
      + 2\kappa^{-2}\norm{A_\mu u}_{\LtwoM}^2 \\
    &\leq \left( \max\{2\kappa^{-2},\, 1+2\kappa^{-2}\norm{A_\mu^0}_{\Linfty}^2 \right)\norm{u}_{\paramGraphSpace}^2.
  \end{align*}
  The proof for the ultraweak formulation is almost identical as $A^*_\mu$ is itself a
  Friedrichs' operator, resulting only in slightly different constants.
  Finally, let us note that all equivalence constants continuously depend on the parameter $\mu$.
  Due to the compactness of the parameter set $\Pcal$, they can thus be uniformly bounded from
  above or below, respectively.
\end{proof}

\begin{theorem}[Exponential approximation of Friedrichs' systems]\label{thm:FSMOR:zerothOrderParam}
	Let $A_\mu$ be a Friedrichs' operator fulfilling~\cref{ass:zerothOrder} with a parameter-separable coefficient $\ahat$. Further, assume parameter-separability of the zeroth-order coefficient $A_\mu^0$. Finally, let the boundary operators $D-M$ and $D+M^*$ be parameter-independent. Then,~\cref{thm:FSMOR:kolmogorov} applies to both the weak and the ultraweak Friedrichs' system.
\end{theorem}
\begin{proof}
	Follows directly from~\cref{lemma:friedrichsDenseness,lemma:friedrichsNorm}.
\end{proof}

\begin{remark}
	Note, that~\cref{ass:zerothOrder} does not directly imply parameter-independence of the boundary operators. As $M$ is not necessarily unique, one could choose different boundary conditions depending on $\mu$, even for a parameter-independent $D$ (which follows e.g.\ for $\ahat=1$).
\end{remark}


\subsubsection*{Exemplary classification of various Friedrichs' systems}
Using~\cref{thm:FSMOR:zerothOrderParam}, many concrete examples of parametrized Friedrichs' systems can already be
classified as exponentially approximable. Following, we list some of these, both in their commonly stated form and in equivalent Friedrichs' form.
For all examples, the following shall be assumed:
\begin{itemize}
	\item All data functions continuously depend on the parameter.
	\item The domain $\Omega$, the parametrization and the boundary operators are chosen in a way such that~\cref{ass:trialSpaceDenseness,ass:testSpaceDenseness} hold.
\end{itemize}
In order to avoid confusion with some of the physical parameters, the parametrized data functions will be highlighted by a subscript $p$.

\newcommand{\param}{p}
\noindent\textbf{Advection-reaction}\\
This example was already discussed earlier and is the only scalar-valued Friedrichs' system. For an advection field $\vec{b}_p$ and a reaction coefficient $c_p$ fulfilling $c_\param+\thalf\nabla\cdot\vec{b}_p\geq\kappa>0$, and a source term $f_p$ the problem reads
\begin{equation*}
	\nabla\cdot(\vec{b}_\param u) + c_\param u = {f_\param},
\end{equation*}
or equivalently in Friedrichs form
\begin{equation*}
	(c_\param+\nabla\cdot\vec{b}_p) u
	+ \sum_{i=1}^d (\vec{b}_{\param})_i \frac{\partial u}{\partial x_i} = {f_\param}.
\end{equation*}
\Cref{thm:FSMOR:zerothOrderParam} can be applied if the advection field $\vec{b}$ is either parameter-independent or solely scales in magnitude and it is known (c.f.\ \cite{OhlbergerRave}) that the general case only gives algebraic decay rates. This example has also been extensively discussed in~\cite{renelt2023model}.

\noindent\textbf{Convection-diffusion-reaction}\\
The convection-diffusion-reaction equation with positive definite diffusivity tensor $D_\param$, given as
\begin{equation*}
    -\nabla\cdot({D_\param}\nabla u) + {\vec{b}_p}\nabla u
    + {c_\param}u = {f_\param}
\end{equation*}
can be transformed into a Friedrichs' system by introducing the total flux $\sigma\coloneqq -D_\param\nabla u + \vec{b}_\param u$, i.e.\
\begin{equation*}
\begin{pmatrix}
	D_\param^{-1} & -D_\param^{-1}\vec{b}_\param \\
	0 & c_\param - \nabla\cdot\vec{b}_\param
\end{pmatrix}
\begin{pmatrix}
	\sigma \\ u
\end{pmatrix}
+
\begin{pmatrix}
	0 & \nabla \\
	\nabla\cdot & 0
\end{pmatrix}
\begin{pmatrix}
\sigma \\ u
\end{pmatrix}
=
\begin{pmatrix}
	0 \\ f_\param
\end{pmatrix}.
\end{equation*}
Thus,~\cref{thm:FSMOR:zerothOrderParam} can be applied without further restrictions resulting in the well-known exponential rates for elliptic problems.


\noindent\textbf{Time-harmonic Maxwell equations}\\
  The time-harmonic Maxwell equations are derived from the full Maxwell equations by assuming time-periodicity with a given frequency $\omega_\param$ (see e.g.~\cite{kirsch2015maxwell}). For magnetic permeability $\mu_\param$, electric permittivity $\varepsilon_\param$ and electric conductivity $\sigma_\param$, and $J_\param$ (the Fourier-transform of) the external electric current, we have
  \begin{equation*}
    \begin{cases}
      -i{\omega_p}{\mu_\param}B + \nabla\times E &= 0 \\
      (-i{\omega_p}{\varepsilon_\param} + {\sigma_\param})E
      - \nabla\times B &= {J_\param}
    \end{cases}
  \end{equation*}
  or in Friedrichs' form
  \begin{equation*}
  	\begin{pmatrix}
  		-i{\omega_p}{\mu_\param} & 0 \\
  		0 & -i{\omega_p}{\varepsilon_\param} + {\sigma_\param}
  	\end{pmatrix}
  	\begin{pmatrix}
  		B \\ E
  	\end{pmatrix}
  	+
  	\begin{pmatrix}
  		0 & \nabla\times \\
  		-\nabla\times & 0
  	\end{pmatrix}
  	\begin{pmatrix}
  		B \\ E
  	\end{pmatrix}
  	=
  	\begin{pmatrix}
  		0 \\ J_\param
  	\end{pmatrix}.
  \end{equation*}
  Technically, we only considered Friedrichs' operators over the real numbers, however, as for example shown in~\cite{antonic2017complexFS}, most concepts can be easily transferred to complex Friedrichs' systems by identifying $\mathbb{C}^m \cong \R^{2m}$ and considering real and imaginary part separately. We therefore expect a similar result as~\cref{thm:FSMOR:zerothOrderParam} to hold in the complex case which would then apply the time-harmonic Maxwell's equations in presented given form.

\noindent\textbf{Linear elasticity}\\
  \newcommand{\stress}[1]{\thalf\left( \nabla #1 + (\nabla #1)^T\right)}
  The linear elasticity equations can be described in terms of the strain $\varepsilon$, stress $\sigma$ and
  displacement $u$ using the relations
  \begin{equation*}
    \begin{cases}
      \hfill\varepsilon &= \stress{u} \\
      \hfill\sigma &= {\lambda_\param} (\nabla\cdot u)I_d + 2\mu_\param\varepsilon \\
      -\nabla\cdot\sigma &= {f_\param}.
    \end{cases}
  \end{equation*}
  with first and second Lam\'e-constants $\lambda_\param$ and $\mu_\param$, as well as an external force field $f_\param$.
  This can be formulated as a Friedrichs' system by setting $\rho\coloneqq -\lambda_\param(\nabla\cdot u)$, $\utilde\coloneqq~2\mu_\param u$ resulting in
  \begin{equation}
  	\begin{cases}
  		\sigma + \rho I_d - \stress{\utilde} &= 0 \\
  		\tr (\sigma) + (d+\frac{2\mu_\param}{\lambda_\param})\rho &= 0 \\
  		-\thalf\nabla\cdot\left(\sigma+\sigma^T\right) &= f_\param.
  	\end{cases}
  \end{equation} 
  We refer to~\cite[Section~3.2]{ernGuermond2006friedrichs2} for the full derivation, definition of the coefficient matrices $\FScoeff$ and subsequent discussion. Once again,~\cref{thm:FSMOR:zerothOrderParam} can be readily applied and indicates exponential convergence without further restrictions.

\section{Conclusion}\label{sec:conclusion}
We have presented a framework to analyze the solution set of parametrized PDEs involving parameter-dependent trial spaces. Instead of the established image of a solution \emph{manifold} embedded into a Hilbert-space (which no longer applies to parametrized solution spaces), an abstraction to a solution \emph{section} of a Hilbert-bundle is proposed and subsequently discussed. It is shown that a topology on the Hilbert-bundle and a generalized notion of linear approximability (in the sense of Kolmogorov) can be derived entirely from a choice of \emph{admissible sections}, 
leading to the notion of Sectional Kolmogorov $N$-widths.
These sections may be interpreted as a set of functions which continuously depend on the parameter and are (computationally) accessible for building reduced approximation spaces. While for certain parametrizations a canonical choice of these sections exists and ensures compatibility with the established definitions, other choices lead to nonlinear notions of approximability which retains relevance beyond the problem of parameter-dependent solution spaces. In contrast to the established manifold-perspective, the section-perspective allows to retain information on how the nonlinear approximation was constructed and which model assumptions entered.
\par
After presenting the framework we showed that if (in addition to other minor assumptions) a certain norm equivalence holds, we can recover exponential convergence of the Sectional Kolmogorov $N$-width. This approximation result is then applied to Friedrichs' systems which provide a relevant example for variational formulations involving a parameter-dependency in the trial space in the form of varying regularity. By the previously shown result, various examples for Friedrichs' systems such as advection-reaction, the time-harmonic Maxwells equations or the linear elasticity equations can then be easily classified once transformed into their corresponding Friedrichs' formulation.
\par
As this work has been also a first step into a new perspective on nonlinear reduction, an important task will be to leverage the insights to further investigate how these insights may be used.
In addition, the connection to other works on nonlinear approximation should be further investigated. In particular, entropy-based concepts of approximability could also be formulated in the fiber-based framework, making use of the locality.


\section*{Declarations}

\subsection*{Funding}
The authors acknowledge funding by the Deutsche Forschungsgemeinschaft
under Germany’s Excellence Strategy EXC 2044 390685587, Mathematics
M\"unster: Dynamics -- Geometry -- Structure.

\noindent C.E. further acknowledges
support by the German Research Foundation (DFG) -- SPP 2410 Hyperbolic
Balance Laws in Fluid Mechanics: Complexity, Scales, Randomness
(CoScaRa), specifically within project 526031774.

\subsection*{Acknowledgements}
The authors thank Jakob Dittmer for the valuable initial discussions on the differential geometric aspects in~\cref{sec:diffGeoFramework}.

\subsection*{Authors' contributions}
\begin{itemize}\itemsep0.5em
\item C.~Engwer: Conceptualization, Supervision, Writing -- Review \& Editing
\item M.~Ohlberger: Conceptualization, Supervision, Writing -- Review \& Editing
\item L.~Renelt: Conceptualization, Formal analysis, Investigation,
  Methodology,
  Writing -- Original draft, Writing -- Review \& Editing
\end{itemize}

\bibliography{sn-bibliography}

\end{document}

%% file: definitions.tex
\usepackage[colorinlistoftodos]{todonotes}

\usepackage{amsthm}%
\usepackage[capitalize,sort&compress,noabbrev]{cleveref} 
\crefformat{equation}{\textup{#2(#1)#3}}
\crefrangeformat{equation}{\textup{#3(#1)#4--#5(#2)#6}}
\crefmultiformat{equation}{\textup{#2(#1)#3}}{ and \textup{#2(#1)#3}}
{, \textup{#2(#1)#3}}{, and \textup{#2(#1)#3}}
\crefrangemultiformat{equation}{\textup{#3(#1)#4--#5(#2)#6}}%
{ and \textup{#3(#1)#4--#5(#2)#6}}{, \textup{#3(#1)#4--#5(#2)#6}}{, and \textup{#3(#1)#4--#5(#2)#6}}

\Crefformat{equation}{#2Equation~\textup{(#1)}#3}
\Crefrangeformat{equation}{Equations~\textup{#3(#1)#4--#5(#2)#6}}
\Crefmultiformat{equation}{Equations~\textup{#2(#1)#3}}{ and \textup{#2(#1)#3}}
{, \textup{#2(#1)#3}}{, and \textup{#2(#1)#3}}
\Crefrangemultiformat{equation}{Equations~\textup{#3(#1)#4--#5(#2)#6}}%
{ and \textup{#3(#1)#4--#5(#2)#6}}{, \textup{#3(#1)#4--#5(#2)#6}}{, and \textup{#3(#1)#4--#5(#2)#6}}

\crefdefaultlabelformat{#2\textup{#1}#3}

\crefname{algo}{Algorithm}{Algorithms}

\crefname{assumption}{Assumption}{Assumptions}
\Crefname{assumption}{Assumption}{Assumptions}

\crefname{example}{Example}{Examples}
\Crefname{example}{Example}{Examples}

\theoremstyle{plain}
\newtheorem{lemma}{Lemma}[section]

\newcounter{exampleSection}[section]

\newcommand{\proj}{\operatorname{P}}

\newcommand{\R}{\mathbb{R}}

\usepackage{amsopn}

\newcommand{\BNB}{Banach-Ne{\v c}as-Babu{\v s}ka}

\newcommand{\utilde}{\tilde{u}}

\newcommand{\Gin}{{\Gamma_{\mathrm{in}}}}
\newcommand{\GinMu}{{\Gamma_{\mathrm{in}}^\mu}}

\newcommand{\norm}[1]{\|#1\|}

\newcommand{\tripleNorm}[1]{{\left\vert\kern-0.25ex\left\vert\kern-0.25ex\left\vert #1
      \right\vert\kern-0.25ex\right\vert\kern-0.25ex\right\vert}}

\newcommand{\opNorm}[1]{\norm{#1}_{\mathrm{op}}}

\newcommand{\Cinf}[1][\Omega]{C^\infty(#1)}
\newcommand{\CinfM}[1][\Omega]{\Cinf[#1]^m}
\newcommand{\Ccinf}[1][\Omega]{C_c^\infty(#1)}
\newcommand{\CcinfM}[1][\Omega]{\Ccinf[#1]^m}

\newcommand{\Ltwo}[1][\Omega]{L^2(#1)}
\newcommand{\LtwoM}[1][\Omega]{\Ltwo[#1]^m}
\newcommand{\Linfty}[1][\Omega]{L^\infty(#1)}

\newcommand{\FSop}{A_\mu}
\newcommand{\FScoeff}[1][i]{A_\mu^{#1}}
\newcommand{\FSadjOp}{\FSop^*}
\newcommand{\graphSpace}[1][\Omega]{H(\FSop;#1)}
\newcommand{\graphSubspace}[1][\Omega]{H_0(\FSop;#1)}
\newcommand{\graphSpaceAdj}[1][\Omega]{H(\FSadjOp;#1)}
\newcommand{\graphSubspaceAdj}[1][\Omega]{H_0(\FSadjOp;#1)}

\newcommand{\FSboundaryD}{D_\mu}
\newcommand{\FSboundaryM}{M_\mu}

\newcommand{\paramGraphSpace}[1][\Omega]{H(A_\mu)}
\newcommand{\paramGraphSubspace}[1][\Omega]{H_0(A_\mu)}
\newcommand{\paramGraphSpaceAdj}[1][\Omega]{H(A^*_\mu)}
\newcommand{\paramGraphSubspaceAdj}[1][\Omega]{H_0(A^*_\mu)}

\newcommand{\xcalBarZ}{\overline{\Xinter}}

\newcommand{\ycalBarZ}{\overline{\Yinter}}

\newcommand{\kolmogorov}[2][N]{\textnormal{d}_{#1}(#2)}
\newcommand{\kolmogorovSec}[2][\Gamma]{\textnormal{d}_{N}(#2; #1)}

\newcommand{\tr}{\operatorname{tr}}

\newcommand{\linSpan}{\operatorname{span}}
\newcommand{\Inf}{\operatornamewithlimits{\vphantom{p}inf}} 
\newcommand{\clos}{\operatorname{clos}}

\usepackage{stmaryrd}

\usepackage{centernot}

\newcommand{\thalf}{\tfrac{1}{2}}

\newcommand{\FindT}{\textnormal{Find}\;}

\newcommand{\forallT}{\textnormal{for all}\;}
\newcommand{\andT}{\;\textnormal{and}\;}
\newcommand{\onT}{\;\textnormal{on}\;}

\newcommand{\diff}{\mathop{}\!\mathrm{d}}

\newcommand{\Mcal}{\mathcal{M}}
\newcommand{\Pcal}{\mathcal{P}}

\newcommand{\Xhat}{\hat{X}}
\newcommand{\neighbourhood}{U(\mu_0)}

\newcommand{\projOne}{\operatorname{proj}_{1}}

\newcommand\restr[1]{\raisebox{-.5ex}{$|$}_{#1}}

\newcommand{\eq}{\;=\;}

\newcommand{\totalSpace}{\Xhat}
\newcommand{\trialSpace}{X_\mu}
\newcommand{\testSpace}{Y_\mu}
\newcommand{\solSection}{\sigma_{\textnormal{sol}}}
\newcommand{\Gconst}{\Gamma_X}
\newcommand{\Xinter}{X_0}
\newcommand{\Yinter}{Y_0}
\newcommand{\ahat}{\hat{a}_{\mu}}